\documentclass[11pt]{amsart}
\usepackage[utf8]{inputenc}
\usepackage[T1]{fontenc}
\usepackage{lmodern}
\usepackage{amsmath,amssymb,amsfonts}
\usepackage[a4paper,margin=1in]{geometry}
\usepackage{microtype}
\usepackage[colorlinks,linkcolor=blue,citecolor=blue,urlcolor=blue]{hyperref}

\newtheorem{theorem}{Theorem}[section]
\newtheorem{proposition}[theorem]{Proposition}
\newtheorem{lemma}[theorem]{Lemma}
\newtheorem{corollary}[theorem]{Corollary}
\theoremstyle{definition}
\newtheorem{definition}[theorem]{Definition}
\newtheorem{example}[theorem]{Example}
\newtheorem{remark}[theorem]{Remark}
\newtheorem{problem}[theorem]{Problem}

\DeclareMathOperator{\Span}{span}
\DeclareMathOperator{\val}{val}
\DeclareMathOperator{\diam}{diam}
\DeclareMathOperator{\supp}{supp}
\DeclareMathOperator{\dist}{dist}
\DeclareMathOperator{\rank}{rank}

\newcommand{\K}{\mathbb{K}}
\newcommand{\R}{\mathbb{R}}
\newcommand{\Z}{\mathbb{Z}}
\newcommand{\F}{\mathbb{F}}

\begin{document}

\title{Expander evolution algebras}

\author{Piero Giacomelli}
\address{IT Department, Tenax Group, Via I Maggio 226, 37020 Volargne (VR), Italy}
\email{pgiacome@gmail.com}
\urladdr{https://orcid.org/0000-0001-7709-3338}


\begin{abstract}
Let $E$ be a finite-dimensional evolution algebra
with natural basis $\mathcal B$, and let $h$ be the Cheeger constant of
the undirected graph of the nonzero off-diagonal structure constants.  We
study how $h$ controls the squaring dynamics, the ideals, the
distinguished idempotent and the tensor products of $E$.  Over $\R$ with
nonnegative structure constants we prove that the support of the $k$-th
plenary power of every basis element is the ball of radius $k$ in the
graph if and only if the diagonal is positive and the support of the
structural matrix is symmetric; expansion then gives exponential support
growth and a bound of order $(d/h)\log n$ on the number of generations
needed to reach the whole basis.  We show that $e_1+\dots+e_n$ is a
primitive idempotent of every symmetric Markov evolution algebra with
connected graph, with Peirce spectrum controlled by $h$ in the regular
case, and we give lower bounds for the Cheeger constant of tensor
products and tensor powers that depend on whether the diagonals vanish.
\end{abstract}

\maketitle

\section{Introduction}\label{sec:intro}

Evolution algebras were introduced by Tian and Vojt\v{e}chovsk\'y
\cite{TianV06,Tian08} to model non-Mendelian inheritance.  They are
commutative, in general nonassociative, algebras with a natural basis
$e_1,\dots,e_n$ in which distinct basis elements multiply to zero, and
$e_i^2=\sum_ja_{ij}e_j$.  The pattern of nonzero structure constants is
recorded by a directed graph, and this graph-theoretic viewpoint has been
used to study ideals, simplicity, nilpotency and its index, automorphisms
and evolution operators; see, for example,
\cite{EL13,EL16,EL19,Casas10,CLOR13,Camacho10,Cadavid19,CSV16,CMMT24}.
In the nilpotent theory the graph decides when products of basis elements
eventually vanish.  The present paper studies the opposite regime, in
which supports grow, and asks the following question:
\emph{how quantitatively does the expansion of the graph control the
algebra?}

We measure expansion by the Cheeger constant $h(E,\mathcal B)$ of the
underlying undirected graph $\varGamma(E,\mathcal B)$ obtained by
forgetting the directions of the nonzero off-diagonal structure
constants (Definition~\ref{def:assoc-graph}).  For a single algebra,
$h(E,\mathcal B)>0$ only means that $\varGamma(E,\mathcal B)$ is
connected; the quantitative content of our results lies in the size of
$h$ relative to the maximum valency $d$ and the dimension $n$, and the
asymptotic content lies in families of algebras with
$\dim E_m\to\infty$ and $h(E_m,\mathcal B_m)$ bounded away from zero
(Definition~\ref{def:EEA} and Remark~\ref{rem:terminology}).  For
nonsingular algebras the graph, and hence $h$, does not depend on the
natural basis (Theorem~\ref{thm:basis-invariance}).

\subsection{Main results}
Our main contributions are the following.

\smallskip
\noindent\textbf{(A) Plenary powers are graph balls.}  Over $\R$ with
nonnegative structure constants, the support of the plenary power
$e_i^{[k]}$ is the closed ball of radius $k$ about $i$ in
$\varGamma(E,\mathcal B)$ for all $i$ and $k$ if and only if every
$a_{ii}$ is positive and the support of the structural matrix is
symmetric (Theorem~\ref{thm:support-growth} and
Proposition~\ref{prop:P-necessary}).  Consequently the Cheeger constant
gives exponential growth of the supports of the plenary powers and the
bound $(2d/h)\log_2n+1$ on the number of generations after which every
basis element is present in the offspring of every other one.  The
combinatorial input is a classical ball-growth estimate, isolated in
Lemma~\ref{lem:ball-growth}; the algebraic input is the absence of
cancellation, and Example~\ref{ex:support-fail} shows that each
hypothesis is needed.

\smallskip
\noindent\textbf{(B) The idempotent $e_1+\dots+e_n$.}  The
multiplication operator of $u=e_1+\dots+e_n$ is the evolution operator.
For symmetric Markov evolution algebras with connected graph, $u$ is a
primitive idempotent in the sense of axial algebras \cite{HRS15a,HRS15b},
and for regular graphs with constant weights the Cheeger constant bounds
the gap between the eigenvalue $1$ of its adjoint action and the rest of
the spectrum (Proposition~\ref{prop:u}).  Axial algebras are currently
studied as commutative algebras generated by idempotents with a fusion
law \cite{GS26}, and the train algebras of rank $3$, another class of
genetic algebras, have been characterised in these terms \cite{GMS24}.

\smallskip
\noindent\textbf{(C) Tensor products.}  The directed graph of
$E_1\otimes E_2$ is the categorical product of the graphs of the factors
\cite[Proposition~5.3]{CMMT23}.  We show that the diagonals decide
whether expansion survives: with zero diagonals the tensor product of two
expanders may be disconnected, while with nonzero diagonals the graph is
the strong product and we give explicit lower bounds for the Cheeger
constant of tensor products and tensor powers
(Theorems~\ref{thm:tensor-zero} and~\ref{thm:tensor-unit}).  The
connectivity statements are special cases of \cite[Theorem~5.4 and
Corollary~5.5]{CMMT23}; the quantitative bounds are new.

\smallskip
We also record the ideal lattice of an evolution algebra with strongly
connected graph (Proposition~\ref{thm:ideals}).  That every nonzero ideal
contains $E^2$, and the resulting simplicity criterion, are known
\cite[Corollaries~3.7 and~4.10]{CSV16}, \cite[Lemma~3.2 and
Proposition~3.4]{CMMT24}; what we add is the explicit description of all
ideals when $E$ is singular.  The spectral statements of
Section~\ref{sec:regular} (Cheeger, Alon--Boppana, Ramanujan) are
classical facts about regular graphs, which we state as corollaries in
the language of evolution algebras; we claim no novelty for them.

The paper is organized as follows.  Section~\ref{sec:prelim} collects
preliminaries.  Section~\ref{sec:EEA} introduces expander evolution
algebras and treats connectivity, ideals and basis invariance.
Section~\ref{sec:structure} contains result (A).
Section~\ref{sec:evop} treats the evolution operator, Markov algebras,
lazy mixing and result (B).  Section~\ref{sec:regular} records the
regular spectral corollaries.  Section~\ref{sec:constructions} gives
Cayley constructions, result (C) and examples, and
Section~\ref{sec:open} lists open problems.

\section{Preliminaries}\label{sec:prelim}

Throughout, $\K$ is a field and all algebras and graphs are finite.  A
\emph{simple graph} is an undirected graph without loops or multiple
edges, $\val(i)$ is the valency of a vertex $i$, $\dist$ is the graph
distance, and $B_r(i)=\{j:\dist(i,j)\le r\}$ is the closed ball of
radius $r$ about $i$.

\subsection{Cheeger constant}
For a graph $\varGamma$ with vertex set $V$ and $W\subseteq V$, let $E(W)$
be the set of edges with exactly one endpoint in $W$.  The (edge)
\emph{Cheeger constant} of $\varGamma$ is
\begin{equation}\label{eq:cheeger-def}
h(\varGamma)=\min\Bigl\{\frac{|E(W)|}{|W|}\;:\;
  \varnothing\ne W\subsetneq V,\ |W|\le\tfrac12|V|\Bigr\}.
\end{equation}

\begin{proposition}\label{prop:cheeger-basic}
Let $\varGamma$ be a graph on $n\ge 2$ vertices with maximum valency at
most $d$.  Then $h(\varGamma)\le d$, and $h(\varGamma)>0$ if and only if
$\varGamma$ is connected, in which case $h(\varGamma)\ge 2/n$.
\end{proposition}

\begin{proof}
Every vertex of $W$ is an endpoint of at most $d$ edges of $E(W)$, so
$|E(W)|\le d|W|$.  If $\varGamma$ is disconnected, a connected component
$C$ of smallest cardinality satisfies $|C|\le n/2$ and $E(C)=\varnothing$.
If $\varGamma$ is connected, then $|E(W)|\ge 1$ for every
$\varnothing\ne W\subsetneq V$, and $|W|\le n/2$ gives
$|E(W)|/|W|\ge 2/n$.
\end{proof}

For a connected $d$-regular graph with adjacency matrix $G$ and
eigenvalues $d=\mu_1\ge\mu_2\ge\dots\ge\mu_n$ we shall use the discrete
Cheeger inequality \cite{AlonMilman85}, \cite[Theorem~2.4]{HLW06}:
\begin{equation}\label{eq:cheeger}
  \frac{d-\mu_2}{2}\;\le\;h(\varGamma)\;\le\;\sqrt{2d\,(d-\mu_2)} .
\end{equation}
The left-hand inequality uses only regularity, and we recall its proof
because it is applied to product graphs below: for
$\varnothing\ne W\subseteq V$ with $|W|\le n/2$ let
$f=\mathbf 1_W-(|W|/n)\mathbf 1$, which is orthogonal to $\mathbf 1$;
then $f^{\top}(dI-G)f=\sum(f_x-f_y)^2=|E(W)|$, the sum running over the
edges $\{x,y\}$, and $\|f\|_2^2=|W|(1-|W|/n)\ge|W|/2$, so the Rayleigh
quotient characterisation of $\mu_2$ gives $d-\mu_2\le 2|E(W)|/|W|$.
Background on expander graphs may be found in \cite{HLW06,Kowalski19}.

\subsection{Evolution algebras}
An \emph{evolution algebra} over $\K$ is a commutative algebra $E$ with a
\emph{natural basis} $\mathcal B=\{e_1,\dots,e_n\}$ such that
\[
e_i e_j=0\quad(i\neq j),\qquad e_i^2=\sum_{j=1}^n a_{ij}e_j .
\]
The matrix $A=(a_{ij})$ is the \emph{structural matrix} of $E$ in
$\mathcal B$.  We write $E^2=\Span\{e_1^2,\dots,e_n^2\}$, so that
$\dim E^2=\rank A$.  The algebra is \emph{simple} when $E^2\neq0$ and
its only ideals are $0$ and $E$, and \emph{nonsingular} when $A$ is
invertible, equivalently $E^2=E$.

For $S\subseteq\{1,\dots,n\}$ set $E_S=\Span\{e_i:i\in S\}$.  Because
products of distinct natural basis vectors vanish, $E_S$ is a subalgebra
precisely when $a_{ij}=0$ for every $i\in S$ and $j\notin S$; in that
case it is automatically an ideal.  We say that $(E,\mathcal B)$ is
\emph{decomposable} if there is a partition $\{1,\dots,n\}=S\sqcup T$
with $S,T\neq\varnothing$ such that $E_S$ and $E_T$ are subalgebras, and
\emph{indecomposable} otherwise; this is the connectedness of
\cite[Chapter~3]{Tian08} read in the fixed natural basis $\mathcal B$.

\subsection{Plenary powers and the evolution operator}
For $x\in E$ the \emph{plenary powers} are defined by $x^{[0]}=x$ and
$x^{[k+1]}=x^{[k]}x^{[k]}$; they describe the successive generations of
the population encoded by $x$ and are not the iterates of a linear map.
The \emph{evolution operator} is the linear map $\mathcal L:E\to E$ with
$\mathcal L(e_i)=e_i^2$; with the row-vector convention its matrix in
$\mathcal B$ is $A$.  If $x=\sum_i\alpha_ie_i$, then
$x^2=\mathcal L(x^{\circ2})$ with $x^{\circ2}=\sum_i\alpha_i^2e_i$
(Lemma~\ref{lem:plenary-coords}), so the plenary powers alternate a
coordinatewise squaring with $\mathcal L$.  The two sequences
$x^{[k]}$ and $\mathcal L^k(x)$ differ in general, even for structure
constants in $\{0,1\}$ (Remark~\ref{rem:plenary-vs-evop}).  A basis
element $e_i$ is \emph{algebraically persistent} when the coefficient of
$e_i$ in $e_i^{[k]}$ is nonzero for every $k\ge0$, and
\emph{algebraically transient} otherwise \cite{Tian08}.

\subsection{Graphs and symmetry conditions}
The directed graph $D(E,\mathcal B)$ has vertex set $\{1,\dots,n\}$ and
an arrow $i\to j$ whenever $i\neq j$ and $a_{ij}\neq0$; this is the graph
underlying the results of \cite{EL13,EL19}.  We distinguish two symmetry
conditions:
\begin{itemize}
  \item $E$ is \emph{symmetric} (with respect to $\mathcal B$) if
        $a_{ij}=a_{ji}$ for all $i,j$, that is, $A$ is a symmetric matrix;
  \item the \emph{support of $A$ is symmetric} if $a_{ij}\neq0$ if and
        only if $a_{ji}\neq0$ for $i\neq j$.
\end{itemize}
The first condition implies the second, but not conversely.  We say that
$E$ is an \emph{adjacency evolution algebra} if $A$ is the adjacency
matrix of a simple graph, that is, $A$ is symmetric with entries in
$\{0,1\}$ and zero diagonal.

\section{Expander evolution algebras}\label{sec:EEA}

\subsection{The underlying undirected graph}

\begin{definition}\label{def:assoc-graph}
Let $E$ be an evolution algebra over $\K$ with natural basis
$\mathcal{B}=\{e_1,\dots,e_n\}$ and structural matrix $A=(a_{ij})$.  The
\emph{underlying graph} $\varGamma(E,\mathcal{B})$ is the simple graph
with vertex set $\{1,\dots,n\}$ in which $\{i,j\}$ is an edge if and
only if $i\ne j$ and $a_{ij}\ne 0$ or $a_{ji}\ne 0$.  We put
$h(E,\mathcal{B})=h(\varGamma(E,\mathcal{B}))$.
\end{definition}

Thus $\varGamma(E,\mathcal{B})$ is the symmetrization of
$D(E,\mathcal{B})$.  It depends only on the off-diagonal support of $A$:
the diagonal entries $a_{ii}$ contribute no edges, although they have a
decisive effect on plenary dynamics (Section~\ref{sec:structure}) and on
tensor products (Section~\ref{sec:tensor}).  If the support of $A$ is
symmetric, then $D(E,\mathcal{B})$ is obtained from
$\varGamma(E,\mathcal{B})$ by replacing each edge with a pair of
opposite arrows, so $D(E,\mathcal{B})$ is strongly connected if and only
if $\varGamma(E,\mathcal{B})$ is connected.

\begin{definition}\label{def:EEA}
Let $h>0$.  We call $(E,\mathcal{B})$ an \emph{$h$-expander evolution
algebra}, or \emph{$h$-EEA} for short, if $h(E,\mathcal{B})\ge h$.  A
sequence $\{(E_m,\mathcal{B}_m)\}_{m\ge1}$ is an \emph{EEA family} if
$\dim E_m\to\infty$ and $\inf_{m}h(E_m,\mathcal{B}_m)>0$.
\end{definition}

\begin{remark}\label{rem:terminology}
By Proposition~\ref{prop:cheeger-basic}, every evolution algebra with
connected underlying graph is an $h$-EEA for some $h>0$, so for a single
algebra the definition is a normalisation rather than a restriction: it
fixes the parameter $h$ that enters all our estimates, typically through
the ratio $d/h$.  The notion corresponding to expander graphs in the
usual asymptotic sense is that of an EEA family with bounded maximum
valency; examples are given in Section~\ref{sec:cayley}.
\end{remark}

\subsection{Connectivity and ideals}\label{sub:connectivity}

\begin{theorem}\label{thm:connectivity}
Let $E$ be an evolution algebra over $\K$ with natural basis
$\mathcal{B}=\{e_1,\dots,e_n\}$, $n\ge 2$.  The following are
equivalent:
\begin{enumerate}
  \item $(E,\mathcal{B})$ is indecomposable;
  \item $\varGamma(E,\mathcal{B})$ is connected;
  \item $h(E,\mathcal{B})>0$.
\end{enumerate}
\end{theorem}

\begin{proof}
The equivalence of (ii) and (iii) is Proposition~\ref{prop:cheeger-basic}.
Let $\{1,\dots,n\}=S\sqcup T$ with $S,T\ne\varnothing$.  By the criterion
recalled in Section~\ref{sec:prelim}, $E_S$ and $E_T$ are both
subalgebras if and only if $a_{ij}=0$ and $a_{ji}=0$ for all $i\in S$ and
$j\in T$, that is, if and only if no edge of $\varGamma(E,\mathcal{B})$
joins $S$ to $T$.  Hence $(E,\mathcal{B})$ is decomposable if and only if
$\varGamma(E,\mathcal{B})$ is disconnected.
\end{proof}

In the following proposition, statement (iv) is
\cite[Corollary~4.10]{CSV16}, and (i) follows from
\cite[Corollary~3.7]{CSV16}, which describes the ideal generated by
$e_k^2$ as the span of the squares $e_j^2$ with $j$ equal to $k$ or a
descendant of $k$, or from \cite[Lemma~3.2 and Proposition~3.4]{CMMT24}.
The description (iii) of the whole ideal lattice, which is informative
only when $E$ is singular, is the part we add; we include the short
argument for completeness.

\begin{proposition}\label{thm:ideals}
Let $E$ be an evolution algebra over $\K$ with natural basis
$\mathcal{B}=\{e_1,\dots,e_n\}$, $n\ge 2$, and structural matrix $A$.
Assume that $D(E,\mathcal{B})$ is strongly connected; this holds
whenever the support of $A$ is symmetric and $h(E,\mathcal{B})>0$.
Then:
\begin{enumerate}
  \item every nonzero ideal of $E$ contains $E^2$;
  \item $E^2\ne 0$, and $E^2$ is the unique minimal nonzero ideal of $E$;
  \item a nonzero subspace $I\subseteq E$ is an ideal if and only if
        $E^2\subseteq I$; consequently the lattice of nonzero ideals of
        $E$ is isomorphic to the lattice of subspaces of $E/E^2$, a
        space of dimension $n-\rank A$;
  \item $E$ is simple if and only if $E^2=E$, if and only if
        $\det A\ne 0$.
\end{enumerate}
\end{proposition}

\begin{proof}
(i) Let $I$ be a nonzero ideal and let $x=\sum_k x_ke_k\in I$ with
$x\ne 0$.  Choose $i$ with $x_i\ne 0$.  Since $e_ke_i=0$ for $k\ne i$,
we have $xe_i=x_ie_i^2\in I$, hence $e_i^2\in I$.  Suppose now that
$e_j^2\in I$ and that $j\to l$ is an arrow of $D(E,\mathcal{B})$, so
$a_{jl}\ne 0$ with $l\ne j$.  Then
\[
  e_j^2\,e_l=\Bigl(\sum_m a_{jm}e_m\Bigr)e_l=a_{jl}\,e_l^2\in I,
\]
hence $e_l^2\in I$.  Since $D(E,\mathcal{B})$ is strongly connected,
every $l$ is reached from $i$ by a directed path, and induction along
this path gives $e_l^2\in I$.  Therefore $E^2\subseteq I$.

(ii) As $n\ge 2$ and $D(E,\mathcal{B})$ is strongly connected, every
vertex $i$ is the tail of some arrow $i\to j$; then $a_{ij}\ne 0$ and
$e_i^2\ne 0$, so $E^2\ne 0$.  The subspace $E^2$ is an ideal, because
$E\cdot E^2\subseteq E^2$.  If $J$ is a nonzero ideal contained in
$E^2$, then $E^2\subseteq J$ by (i), so $J=E^2$.  If $M$ is any minimal
nonzero ideal, then $E^2\subseteq M$ by (i), and minimality gives
$M=E^2$.

(iii) If $I$ is a nonzero ideal, then $E^2\subseteq I$ by (i).
Conversely, if $V$ is a subspace containing $E^2$, then
$E\cdot V\subseteq E^2\subseteq V$, so $V$ is an ideal.  The map
$V\mapsto V/E^2$ is an isomorphism from the lattice of subspaces of $E$
containing $E^2$ onto the lattice of subspaces of $E/E^2$, and
$\dim E/E^2=n-\rank A$.

(iv) If $E^2=E$, then by (iii) $E$ is the only nonzero ideal, so $E$ is
simple.  If $E^2\ne E$, then $E^2$ is an ideal with $0\ne E^2\ne E$ by
(ii), so $E$ is not simple.  Finally $E^2=E$ if and only if
$\rank A=n$.
\end{proof}

\begin{example}\label{ex:path}
Let $E$ have natural basis $\{e_1,e_2,e_3\}$ and
$e_1^2=e_2$, $e_2^2=e_1+e_3$, $e_3^2=e_2$ over any field $\K$.  Then
$E$ is symmetric, $\varGamma(E,\mathcal{B})$ is the path $1-2-3$ and
$h(E,\mathcal{B})=1$, the minimum being attained at $W=\{1\}$.  The
first and third rows of $A$ coincide, so $\det A=0$, and by
Proposition~\ref{thm:ideals}(iii) the ideals of $E$ are exactly $0$,
$E^2=\Span\{e_2,e_1+e_3\}$ and $E$.  Similarly, the symmetric Markov
evolution algebra $e_1^2=e_2^2=\tfrac12(e_1+e_2)$ has $h=1$ and the
proper nonzero ideal $E^2=\K(e_1+e_2)$.
\end{example}

\begin{remark}\label{rem:ideals-hyp}
Strong connectivity of $D(E,\mathcal{B})$ cannot be replaced by
connectivity of $\varGamma(E,\mathcal{B})$.  The algebra with
$e_1^2=e_1+e_2$ and $e_2^2=e_2$ is nonsingular and has
$\varGamma(E,\mathcal{B})$ equal to a single edge, so $h=1$; however
$\K e_2$ is a proper nonzero ideal.  In the language of~\cite{CMMT24},
strong connectivity means that the only hereditary subsets of the vertex
set are $\varnothing$ and the whole set; the same argument is used in
\cite[Example~3.3]{CCGMMR26}.
\end{remark}

\subsection{Basis invariance}\label{sub:basis-inv}

For singular evolution algebras natural bases need not be related by
monomial matrices, and the underlying graph may a priori depend on the
choice of $\mathcal B$ (Problem~\ref{prob:singular-basis}).  For
nonsingular evolution algebras it does not.  This fact is well known; we
include the short proof, since the transformation rule for the
structure constants is used in Remark~\ref{rem:diag-invariance}.

\begin{theorem}\label{thm:basis-invariance}
Let $E$ be a nonsingular evolution algebra over $\K$ and let
$\mathcal{B}=\{e_1,\dots,e_n\}$ and $\mathcal{B}'=\{f_1,\dots,f_n\}$
be natural bases of $E$, with structural matrices $A=(a_{ij})$ and
$A'=(a'_{ij})$.  Then there are a permutation $\sigma$ of
$\{1,\dots,n\}$ and scalars $\lambda_1,\dots,\lambda_n\in\K^\times$
with $f_i=\lambda_ie_{\sigma(i)}$ for all $i$, and
\[
  a'_{il}=\lambda_i^{2}\lambda_l^{-1}\,a_{\sigma(i)\sigma(l)}
  \qquad(1\le i,l\le n).
\]
Consequently $\sigma$ is an isomorphism of $D(E,\mathcal{B}')$ onto
$D(E,\mathcal{B})$ and of $\varGamma(E,\mathcal{B}')$ onto
$\varGamma(E,\mathcal{B})$, and $h(E,\mathcal{B}')=h(E,\mathcal{B})$.
\end{theorem}

\begin{proof}
Write $f_i=\sum_j p_{ij}e_j$ with $P=(p_{ij})$ invertible.  For $i\ne k$,
\[
  0=f_if_k=\sum_{j,m}p_{ij}p_{km}\,e_je_m=\sum_j p_{ij}p_{kj}\,e_j^2 .
\]
As $E$ is nonsingular, $e_1^2,\dots,e_n^2$ are linearly independent,
so $p_{ij}p_{kj}=0$ for every $j$ and all $i\ne k$.  Thus each column
of $P$ has at most one nonzero entry.  Since $P$ is invertible, no
column and no row of $P$ vanishes; hence every column has exactly one
nonzero entry, $P$ has exactly $n$ nonzero entries, and every row has
exactly one.  So $P$ is a monomial matrix: $f_i=\lambda_ie_{\sigma(i)}$
for a permutation $\sigma$ and $\lambda_i\ne 0$.  Then
\[
  f_i^2=\lambda_i^2\,e_{\sigma(i)}^2
  =\lambda_i^2\sum_{l}a_{\sigma(i)\sigma(l)}\,e_{\sigma(l)}
  =\sum_l\lambda_i^2\lambda_l^{-1}a_{\sigma(i)\sigma(l)}\,f_l ,
\]
which is the stated formula.  In particular $a'_{il}\ne 0$ if and only
if $a_{\sigma(i)\sigma(l)}\ne 0$, so $\sigma$ maps arrows and edges of
the graphs attached to $\mathcal{B}'$ bijectively onto those attached to
$\mathcal{B}$, and a graph isomorphism preserves Cheeger constants.
\end{proof}

\begin{remark}\label{rem:diag-invariance}
For nonsingular $E$, Theorem~\ref{thm:basis-invariance} shows that the
Cheeger constant, the property of being an $h$-EEA, indecomposability,
symmetry of the support of $A$ and the set of indices with
$a_{ii}\ne0$ (since $a'_{ii}=\lambda_ia_{\sigma(i)\sigma(i)}$) do not
depend on $\mathcal B$.  Signs of structure constants over $\R$ do
depend on $\mathcal B$: replacing $e_i$ by $-e_i$ changes the sign of
$a_{ii}$ and of the $a_{li}$ with $l\ne i$.  Hence the positivity
hypothesis (P) of Section~\ref{sec:structure} is a property of the pair
$(E,\mathcal B)$, and it is preserved under rescaling by positive
scalars.
\end{remark}

\section{Plenary powers and graph balls}\label{sec:structure}

\subsection{Subalgebra expansion bound}

\begin{proposition}\label{thm:subalgebra-expansion}
Let $(E,\mathcal{B})$ be an $h$-EEA over $\K$, where
$\mathcal{B}=\{e_1,\dots,e_n\}$, and let
$\varnothing\ne S\subsetneq\{1,\dots,n\}$ be such that $E_S$ is a
subalgebra.  Then no arrow of $D(E,\mathcal{B})$ leaves $S$, and the
number of arrows $j\to i$ with $j\notin S$ and $i\in S$ is at least
$|E(S)|\ge h\cdot\min\{|S|,\,n-|S|\}$.
\end{proposition}

\begin{proof}
Since $E_S$ is a subalgebra, $a_{ij}=0$ for $i\in S$ and $j\notin S$,
so no arrow leaves $S$.  Consequently every edge $\{i,j\}\in E(S)$ with
$i\in S$, $j\notin S$ arises from an arrow $j\to i$, and distinct edges
arise from distinct arrows.  If $|S|\le n/2$, the definition of the
Cheeger constant gives $|E(S)|\ge h|S|$; otherwise $|S^c|<n/2$ and
$E(S)=E(S^c)$ give $|E(S)|\ge h\,(n-|S|)$.
\end{proof}

In the genetic interpretation, an ideal spanned by a set $S$ of
genotypes receives at least $h\min\{|S|,n-|S|\}$ direct contributions
from the genotypes outside $S$, while it sends none back.

\subsection{Ball growth and diameter}
The following estimate is classical; we isolate it because it is the
only combinatorial input of Theorem~\ref{thm:support-growth}.

\begin{lemma}\label{lem:ball-growth}
Let $\varGamma$ be a graph on $n$ vertices with maximum valency at most
$d$ and $h(\varGamma)\ge h>0$.  For every vertex $u$ and every $r\ge0$,
\[
  |B_r(u)|\ge\min\bigl\{(1+h/d)^r,\ n/2\bigr\}.
\]
\end{lemma}

\begin{proof}
It suffices to show that $|B_r(u)|\le n/2$ implies
$|B_r(u)|\ge(1+h/d)^r$.  This is clear for $r=0$.  If $r\ge0$ and
$|B_{r+1}(u)|\le n/2$, then $|B_r(u)|\le n/2$, so
$|B_r(u)|\ge(1+h/d)^r$ by induction, and~\eqref{eq:cheeger-def} gives
$|E(B_r(u))|\ge h|B_r(u)|$.  Every edge of $E(B_r(u))$ has its outer
endpoint in $B_{r+1}(u)\setminus B_r(u)$, and each such vertex is an
endpoint of at most $d$ of these edges, so
$|B_{r+1}(u)\setminus B_r(u)|\ge(h/d)|B_r(u)|$ and
$|B_{r+1}(u)|\ge(1+h/d)^{r+1}$.
\end{proof}

\begin{theorem}\label{thm:diameter}
Let $(E,\mathcal B)$ be an $h$-EEA over $\K$ with $n=\dim E\ge2$ and
$\max_i\val(i)\le d$ in $\varGamma(E,\mathcal{B})$.  Then
\[
  \diam\varGamma(E,\mathcal{B})\le\frac{2d}{h}\log_2 n+1 .
\]
\end{theorem}

\begin{proof}
Write $\varGamma=\varGamma(E,\mathcal{B})$ and $V=\{1,\dots,n\}$.  The
function $x\mapsto(1+x)^{1/x}$ is decreasing on $(0,1]$ and equals $2$ at
$x=1$, so $\log_2(1+x)\ge x$ for $x\in(0,1]$; by
Proposition~\ref{prop:cheeger-basic} this applies to $x=h/d$.  Let
$r_0=\lceil\log_{1+h/d}(n/2)\rceil$ if $n\ge 3$ and $r_0=0$ if $n=2$.
Then $(1+h/d)^{r_0}\ge n/2$ and, using $d/h\ge1$,
\[
  r_0\le\frac{\log_2(n/2)}{\log_2(1+h/d)}+1\le
  \frac{d}{h}(\log_2 n-1)+1\le\frac{d}{h}\log_2 n .
\]
By Lemma~\ref{lem:ball-growth}, $|B_{r_0}(u)|\ge n/2$ for every $u\in V$.

Let $u,v\in V$.  If $B_{r_0}(u)\cap B_{r_0}(v)\ne\varnothing$, then
$\dist(u,v)\le 2r_0$.  Otherwise the two balls are disjoint and each has
at least $n/2$ elements, so $n$ is even, each ball has exactly $n/2$
elements, and the two balls partition $V$; for odd $n$ this case does
not occur.  Since $|E(B_{r_0}(u))|\ge hn/2>0$, some edge
$\{\alpha,\beta\}$ has $\alpha\in B_{r_0}(u)$ and $\beta\in B_{r_0}(v)$,
and $\dist(u,v)\le 2r_0+1$.  In both cases
$\dist(u,v)\le(2d/h)\log_2 n+1$.
\end{proof}

\subsection{Supports of plenary powers}
For $x=\sum_k\alpha_ke_k\in E$ put $\supp(x)=\{k:\alpha_k\ne 0\}$, and
identify $x$ with its coordinate row vector $(\alpha_1,\dots,\alpha_n)$.

\begin{lemma}\label{lem:plenary-coords}
Let $x\in E$ have coordinate vector $v=(\alpha_1,\dots,\alpha_n)$ and
put $v^{\circ 2}=(\alpha_1^2,\dots,\alpha_n^2)$.  Then $x^2$ has
coordinate vector $v^{\circ 2}A$, that is, $x^2=\mathcal L(x^{\circ 2})$.
Consequently, if $v_k$ is the coordinate vector of $x^{[k]}$, then
$v_{k+1}=v_k^{\circ 2}A$ for all $k\ge 0$.
\end{lemma}

\begin{proof}
Since $e_je_m=0$ for $j\ne m$,
$x^2=\sum_{j,m}\alpha_j\alpha_m e_je_m=\sum_j\alpha_j^2e_j^2
=\sum_l\bigl(\sum_j\alpha_j^2a_{jl}\bigr)e_l$.
\end{proof}

The support of the plenary powers is governed by the graph only in the
absence of cancellation.  We use the following hypothesis on
$(E,\mathcal{B})$:

\smallskip
\noindent\textbf{(P)} $\K=\R$, $a_{ij}\ge 0$ for all $i,j$, $a_{ii}>0$
for all $i$, and the support of $A$ is symmetric.
\smallskip

\noindent For instance, symmetric Markov evolution algebras with positive
diagonal, in particular the lazy algebras of Section~\ref{sec:evop},
satisfy (P).

\begin{lemma}\label{lem:support}
Assume (P).  If $x\in E$ has nonnegative coordinates, then so does
$x^2$, and $\supp(x^2)$ is the set of vertices of
$\varGamma(E,\mathcal{B})$ at distance at most $1$ from $\supp(x)$.
\end{lemma}

\begin{proof}
By Lemma~\ref{lem:plenary-coords} the coefficient of $e_l$ in $x^2$ is
$\sum_j\alpha_j^2a_{jl}$, a sum of nonnegative terms.  It is positive if
and only if $a_{jl}>0$ for some $j\in\supp(x)$.  For $l=j$ this always
holds, because $a_{jj}>0$.  For $l\ne j$ it holds if and only if
$\{j,l\}$ is an edge of $\varGamma(E,\mathcal{B})$, because the support
of $A$ is symmetric and its entries are nonnegative.
\end{proof}

\begin{theorem}\label{thm:support-growth}
Let $(E,\mathcal B)$ be an $h$-EEA satisfying (P), with
$\max_i\val(i)\le d$ in $\varGamma=\varGamma(E,\mathcal{B})$.  For every
$i$ and every $k\ge 0$:
\begin{enumerate}
  \item $e_i^{[k]}$ has nonnegative coordinates and
        $\supp(e_i^{[k]})=B_k(i)$;
  \item $|\supp(e_i^{[k]})|\ge\min\{(1+h/d)^k,\,n/2\}$;
  \item $e_i^{[k]}$ has full support if and only if $k$ is at least the
        eccentricity $\max_j\dist(i,j)$ of $i$; in particular every
        $e_i^{[k]}$ with $k\ge(2d/h)\log_2 n+1$ has full support.
\end{enumerate}
\end{theorem}

\begin{proof}
(i) By induction on $k$: $\supp(e_i^{[0]})=\{i\}=B_0(i)$, and if
$e_i^{[k]}$ is nonnegative with support $B_k(i)$, then
Lemma~\ref{lem:support} shows that $e_i^{[k+1]}=(e_i^{[k]})^2$ is
nonnegative with support $B_{k+1}(i)$.  (ii) follows from (i) and
Lemma~\ref{lem:ball-growth}.  (iii) By (i), $e_i^{[k]}$ has full support
if and only if $B_k(i)=\{1,\dots,n\}$; the eccentricity is at most
$\diam\varGamma$, which is bounded in Theorem~\ref{thm:diameter}.
\end{proof}

Within nonnegative real structure constants, the hypotheses of (P) are
not only sufficient but necessary.

\begin{proposition}\label{prop:P-necessary}
Let $E$ be an evolution algebra over $\R$ with $a_{ij}\ge0$ for all
$i,j$.  Then $\supp(e_i^{[k]})=B_k(i)$ for all $i$ and all $k\ge0$ if
and only if $(E,\mathcal B)$ satisfies (P).  In fact, already the case
$k=1$, $\supp(e_i^2)=B_1(i)$ for all $i$, implies (P).
\end{proposition}

\begin{proof}
Sufficiency is Theorem~\ref{thm:support-growth}(i), whose proof does not
use $h$.  Conversely, assume $\supp(e_i^2)=B_1(i)$ for all $i$.  Since
$\supp(e_i^2)=\{l:a_{il}\ne0\}$ and $i\in B_1(i)$, we get $a_{ii}\ne0$,
hence $a_{ii}>0$.  If $a_{ji}\ne0$ for some $j\ne i$, then $\{i,j\}$ is
an edge of $\varGamma(E,\mathcal B)$, so $j\in B_1(i)=\supp(e_i^2)$ and
$a_{ij}\ne0$.  Thus the support of $A$ is symmetric.
\end{proof}

\begin{example}\label{ex:support-fail}
Outside nonnegative real structure constants, each hypothesis of (P) is
still needed.
\begin{enumerate}
  \item \emph{Zero diagonal.}  Let $e_j^2=e_{j-1}+e_{j+1}$ (indices
        mod $6$).  The graph is the cycle $C_6$, so $h>0$, but
        $\supp(e_0^{[k]})$ is contained in the set of indices of the
        same parity as $k$; no plenary power of $e_0$ has full support.
  \item \emph{Signs.}  Over $\R$ let $e_1^2=e_1-e_2$ and
        $e_2^2=-e_1+2e_2$.  The algebra is symmetric with positive
        diagonal and $h=1$, but $e_1^{[2]}=e_1^2+e_2^2=e_2$, so
        $\supp(e_1^{[2]})=\{2\}\ne B_2(1)$.
  \item \emph{Characteristic.}  Over a field of characteristic $2$ let
        $e_j^2=e_{j-1}+e_j+e_{j+1}$ (indices mod $4$, in $\{1,2,3,4\}$).
        The structure constants are $0$ or $1$, the diagonal is nonzero
        and the graph is $C_4$, yet $e_1^{[1]}=e_4+e_1+e_2$ and
        $e_1^{[2]}=e_4^2+e_1^2+e_2^2=3e_1+2e_2+2e_3+2e_4=e_1$.
\end{enumerate}
\end{example}

\subsection{Algebraic persistency}

\begin{proposition}\label{prop:persistency}
Let $E$ be an evolution algebra over $\K$ with natural basis
$\mathcal{B}$.
\begin{enumerate}
  \item If $e_i$ is algebraically persistent, then $a_{ii}\ne 0$.
  \item If $\K=\R$ and $a_{jl}\ge 0$ for all $j,l$, then $e_i$ is
        algebraically persistent if and only if $a_{ii}>0$.  In
        particular every generator of an algebra satisfying (P) is
        algebraically persistent.
\end{enumerate}
\end{proposition}

\begin{proof}
(i) The coefficient of $e_i$ in $e_i^{[1]}=e_i^2$ is $a_{ii}$.
(ii) Let $c_k$ be the coefficient of $e_i$ in $e_i^{[k]}$.  By
Lemma~\ref{lem:plenary-coords} and induction on $k$, all coordinates of
$e_i^{[k]}$ are nonnegative, and $c_{k+1}\ge c_k^2a_{ii}$.  If
$a_{ii}>0$, then $c_0=1$ and induction gives $c_k>0$ for all $k$.  The
converse is (i).
\end{proof}

Expansion plays no role in Proposition~\ref{prop:persistency}, and over
$\R$ with entries of both signs the condition $a_{ii}\ne 0$ is not
sufficient even for symmetric EEAs: in
Example~\ref{ex:support-fail}(ii) the coefficient of $e_1$ in
$e_1^{[2]}=e_2$ vanishes, so $e_1$ is algebraically transient.

\section{The evolution operator and Markov algebras}\label{sec:evop}

\subsection{Supports of the evolution operator}

\begin{remark}\label{rem:plenary-vs-evop}
By Lemma~\ref{lem:plenary-coords},
$x^{[k+1]}=\mathcal L((x^{[k]})^{\circ 2})$, while the matrix of
$\mathcal L^k$ in $\mathcal B$ is $A^k$.  For $e_j^2=e_{j-1}+e_{j+1}$
(indices mod $4$) one finds $e_1^{[2]}=\mathcal L^2(e_1)=2(e_1+e_3)$, but
$e_1^{[3]}=8(e_2+e_4)$ while $\mathcal L^3(e_1)=4(e_2+e_4)$.  The two
sequences have the same supports under (P).
\end{remark}

\begin{theorem}\label{thm:evop-distance}
Assume (P).  For every $i$ and $k\ge 0$ one has
$\supp(\mathcal L^k(e_i))=B_k(i)$.  Hence, if $(E,\mathcal B)$ is an
$h$-EEA with $\max_i\val(i)\le d$, the least $k$ with
$\supp(\mathcal L^k(e_i))=\{1,\dots,n\}$ is the eccentricity of $i$ in
$\varGamma(E,\mathcal{B})$, and it is at most $(2d/h)\log_2 n+1$.
\end{theorem}

\begin{proof}
The coordinate vector of $\mathcal L^k(e_i)$ is the $i$-th row of $A^k$,
and $(A^k)_{il}=\sum\prod_{t=1}^k a_{j_{t-1}j_t}$, the sum running over
all sequences $i=j_0,j_1,\dots,j_k=l$, is a sum of nonnegative terms.
Under (P), $a_{jj'}>0$ holds exactly when $j'=j$ or $\{j,j'\}$ is an
edge of $\varGamma(E,\mathcal{B})$, so $(A^k)_{il}>0$ if and only if
there is a walk of length $k$ from $i$ to $l$ in which the walker may
also stay at a vertex.  Such a walk exists if and only if
$\dist(i,l)\le k$.  The last assertion follows from
Theorem~\ref{thm:diameter}.
\end{proof}

The eigenvalues of $\mathcal L$ are the roots of $\det(tI-A)\in\K[t]$ in
an algebraic closure of $\K$; they need not lie in $\K$ (for
$e_1^2=e_1+e_2$, $e_2^2=e_1$ over $\mathbb Q$ the characteristic
polynomial is $t^2-t-1$).  The spectral statements below are therefore
made over $\R$.

\begin{theorem}\label{thm:spectral-radius}
Let $(E,\mathcal B)$ be an $h$-EEA over $\R$ whose structural matrix $A$
is symmetric with nonnegative entries, and let
$\lambda_1\ge\lambda_2\ge\dots\ge\lambda_n$ be its eigenvalues.  Then
$\lambda_1=\rho(A)$, $\lambda_1$ is a simple eigenvalue, and it has an
eigenvector with all coordinates positive.
\end{theorem}

\begin{proof}
By Theorem~\ref{thm:connectivity} the graph $\varGamma(E,\mathcal{B})$
is connected; as $A$ is symmetric, $D(E,\mathcal{B})$ is strongly
connected, that is, $A$ is irreducible.  By the Perron--Frobenius
theorem~\cite[Theorem~8.4.4]{HJ13}, $\rho(A)$ is a simple eigenvalue
with a positive eigenvector, and $\lambda_j\le|\lambda_j|\le\rho(A)$ for
all $j$.
\end{proof}

Nonnegativity cannot be dropped: the symmetric matrix
$\bigl(\begin{smallmatrix}-5&1\\1&0\end{smallmatrix}\bigr)$ defines a
$1$-EEA whose largest eigenvalue $(-5+\sqrt{29})/2$ is smaller than its
spectral radius.  Moreover, no lower bound for $\lambda_1-\lambda_2$ in
terms of $h$ alone can hold, since replacing $A$ by $\varepsilon A$
multiplies every eigenvalue by $\varepsilon>0$ and leaves the graph
unchanged.  Quantitative spectral bounds therefore need control of the
weights; below we assume regular graphs and constant weights.

\subsection{Markov algebras and mixing}

\begin{definition}\label{def:markov-EEA}
An evolution algebra $E$ over $\R$ is a \emph{Markov evolution algebra}
if $a_{ij}\ge 0$ for all $i,j$ and every row of $A$ sums to $1$.  It is
a \emph{symmetric Markov evolution algebra} if moreover $A$ is a
symmetric matrix, so that $A$ is doubly stochastic; this is stronger
than requiring the support of $A$ to be symmetric.  The \emph{lazy
algebra} $E^{\mathrm{lazy}}$ of a Markov evolution algebra $E$ is the
Markov evolution algebra on the same natural basis with structural
matrix $P=\tfrac12(I+A)$.
\end{definition}

Here $\mathcal L$ describes one step of the Markov chain with transition
matrix $A$.  Diagonal entries contribute no edges, so
$\varGamma(E^{\mathrm{lazy}},\mathcal{B})=\varGamma(E,\mathcal{B})$;
moreover $E^{\mathrm{lazy}}$ satisfies (P) whenever the support of $A$ is
symmetric.

\begin{proposition}\label{thm:markov-simple}
Let $E$ be a Markov evolution algebra whose structural matrix has
symmetric support and whose graph $\varGamma(E,\mathcal B)$ is
connected.  Then the Markov chain with transition matrix $A$ is
irreducible, and the ideals of $E$ are described by
Proposition~\ref{thm:ideals}; in particular $E$ is simple if and only if
$\det A\ne 0$.
\end{proposition}

\begin{proof}
Since the support of $A$ is symmetric, strong connectivity of
$D(E,\mathcal{B})$, which is irreducibility of the chain, is equivalent
to connectivity of $\varGamma(E,\mathcal{B})$.
\end{proof}

Irreducibility does not imply simplicity: the Markov algebra
$e_1^2=e_2^2=\tfrac12(e_1+e_2)$ of Example~\ref{ex:path} is irreducible
but has the proper ideal $E^2=\R(e_1+e_2)$.  Without the symmetry of the
support, connectivity of $\varGamma(E,\mathcal{B})$ need not imply
irreducibility: a directed path $1\to 2\to 3$ gives a connected
undirected graph.

The following mixing estimate is a direct consequence
of~\eqref{eq:cheeger}; it requires regularity and constant weights, not
only $h>0$.

\begin{theorem}\label{thm:mixing}
Let $(E,\mathcal B)$ be a Markov $h$-EEA such that
$\varGamma(E,\mathcal{B})$ is $d$-regular, $a_{ij}=1/d$ for each edge
$\{i,j\}$ and $a_{ij}=0$ otherwise.  Let $P=\tfrac12(I+A)$ and
$\pi=(1/n,\dots,1/n)$.  For every generator $e_i$ and every $k\ge 0$,
\[
  \|e_iP^k-\pi\|_1\le\sqrt{n}\left(1-\frac{h^2}{4d^2}\right)^k ,
\]
where $e_iP^k$ is the coordinate vector of the $k$-th iterate of the
evolution operator of $E^{\mathrm{lazy}}$ applied to $e_i$, that is, the
distribution after $k$ steps of the lazy chain started at $i$.
\end{theorem}

\begin{proof}
Write $A=G/d$, where $G$ is the adjacency matrix of the connected
$d$-regular graph $\varGamma(E,\mathcal{B})$, with eigenvalues
$d=\mu_1>\mu_2\ge\dots\ge\mu_n\ge-d$ (Proposition~\ref{thm:spectral-C}).
By~\eqref{eq:cheeger}, $h\le\sqrt{2d(d-\mu_2)}$, so
$d-\mu_2\ge h^2/(2d)$.  The eigenvalues of the symmetric matrix $P$ are
$\tfrac12(1+\mu_j/d)\in[0,1]$; the largest is $1$, with eigenvector
$\mathbf 1$, and all others lie in $[0,\rho]$ with
\[
  \rho=\tfrac12\Bigl(1+\frac{\mu_2}{d}\Bigr)
  =1-\frac{d-\mu_2}{2d}\le 1-\frac{h^2}{4d^2}.
\]
Put $w=e_i-\pi$.  Its coordinates sum to $0$, so $w\perp\mathbf 1$ and
$\|wP^k\|_2\le\rho^k\|w\|_2$.  Since $\pi P=\pi$ we have
$e_iP^k-\pi=wP^k$, and $\|w\|_2^2=1-1/n<1$.  Finally
$\|v\|_1\le\sqrt n\,\|v\|_2$ for every $v\in\R^n$.
\end{proof}

\begin{corollary}\label{cor:mixing-time}
Under the hypotheses of Theorem~\ref{thm:mixing}, the mixing time
\[
  t_\mathrm{mix}(\epsilon)=\min\{k:\max_i\|e_iP^k-\pi\|_1\le\epsilon\}
\]
satisfies
$t_\mathrm{mix}(\epsilon)\le\bigl\lceil(4d^2/h^2)\log(\sqrt n/\epsilon)\bigr\rceil$.
\end{corollary}

\begin{proof}
Since $1-x\le e^{-x}$, Theorem~\ref{thm:mixing} gives
$\|e_iP^k-\pi\|_1\le\sqrt n\,e^{-kh^2/(4d^2)}$.
\end{proof}

The passage to the lazy algebra cannot be avoided: for the $4$-cycle
($d=2$, $h=1$), $e_1A^k$ is supported on $\{2,4\}$ for $k$ odd and on
$\{1,3\}$ for $k$ even, so $\|e_1A^k-\pi\|_1=1$ for all $k\ge 1$.

\subsection{The idempotent $u$}
Put $u=e_1+\dots+e_n$ and let $L_x\colon y\mapsto xy$ denote the
multiplication operator of $x\in E$.  Following \cite{HRS15b}, an
idempotent $a$ of a commutative algebra is \emph{primitive} if the
eigenspace of $L_a$ for the eigenvalue $1$ is $\K a$.

\begin{proposition}\label{prop:u}
Let $E$ be an evolution algebra over $\K$ with natural basis
$\mathcal B=\{e_1,\dots,e_n\}$, $n\ge2$, and structural matrix $A$.
\begin{enumerate}
  \item $L_u=\mathcal L$; in particular the matrix of $L_u$ in
        $\mathcal B$ is $A$.
  \item $u^2=u$ if and only if every column of $A$ sums to $1$.
  \item If $E$ is a symmetric Markov evolution algebra
        (Definition~\ref{def:markov-EEA}) and $\varGamma(E,\mathcal B)$
        is connected, then $u$ is a primitive idempotent, and every
        eigenvalue of $L_u$ is real and lies in $[-1,1]$.
  \item If moreover $\varGamma(E,\mathcal B)$ is $d$-regular and
        $a_{ij}=1/d$ for each edge $\{i,j\}$, as in
        Theorem~\ref{thm:mixing}, then every eigenvalue $\lambda\ne 1$
        of $L_u$ satisfies $\lambda\le 1-h(E,\mathcal B)^2/(2d^2)$.
\end{enumerate}
\end{proposition}

\begin{proof}
(i) For every $i$ we have $ue_i=\sum_je_je_i=e_i^2=\mathcal L(e_i)$.

(ii) We have $u^2=\sum_je_j^2=\sum_l\bigl(\sum_ja_{jl}\bigr)e_l$.

(iii) Since $A$ is symmetric and its rows sum to $1$, its columns sum
to $1$, so $u^2=u$ by (ii).  The matrix $A$ is real symmetric and
nonnegative, $\rho(A)=1$ because $A$ is stochastic, and $A$ is
irreducible because $\varGamma(E,\mathcal B)$ is connected and $A$ is
symmetric.  By (i) and Theorem~\ref{thm:spectral-radius}, the
eigenvalues of $L_u$ are real, lie in $[-1,1]$, and the eigenvalue $1$ is
simple.  Since $L_u(u)=u$, the $1$-eigenspace of $L_u$ is $\R u$.

(iv) Here $A=G/d$ with $G$ the adjacency matrix of
$\varGamma(E,\mathcal B)$, so the eigenvalues of $L_u$ are $\mu_j/d$.
For $\lambda\ne 1$ we have $\lambda\le\mu_2/d$, and~\eqref{eq:cheeger}
gives $d-\mu_2\ge h(E,\mathcal B)^2/(2d)$.
\end{proof}

Symmetry of $A$, not only of its support, is needed in (iii): for
$e_1^2=\tfrac12(e_1+e_2)$, $e_2^2=\tfrac14e_1+\tfrac34e_2$ the support
is symmetric and the graph is connected, but the column sums are
$\tfrac34$ and $\tfrac54$, so $u$ is not idempotent.

Thus for a symmetric Markov evolution algebra with connected graph the
evolution operator is the adjoint action of a primitive idempotent, its
Peirce decomposition with respect to $u$ is the spectral decomposition of
the Markov chain, and the Cheeger constant controls the gap between the
eigenvalue $1$ and the remaining eigenvalues.  Whether the products of
the eigenvectors of $L_u$ obey a fusion law in the sense of
\cite{HRS15a,HRS15b} is the subject of Problem~\ref{prob:fusion}.

\section{Regular spectral corollaries}\label{sec:regular}

The statements of this section are classical facts about regular graphs
\cite{AlonMilman85,HLW06,Nilli91}, translated to evolution algebras; we
record them as corollaries for later use.  An evolution algebra is
\emph{$d$-regular} if $\varGamma(E,\mathcal{B})$ is $d$-regular.

\begin{proposition}\label{thm:spectral-C}
Let $E$ be a $d$-regular adjacency evolution algebra over $\R$ with
connected graph.  Then all eigenvalues of $A$ are real,
$-d\le\lambda_n\le\dots\le\lambda_2<\lambda_1=d$, and the eigenspace of
$d$ is spanned by $(1,\dots,1)$.
\end{proposition}

\begin{proof}
The matrix $A$ is real symmetric and nonnegative, its rows sum to $d$,
so $\rho(A)=d$ with eigenvector $(1,\dots,1)$.  Connectivity gives
irreducibility, and the Perron--Frobenius theorem makes $d$ simple.
\end{proof}

\begin{corollary}\label{thm:cheeger-evop}
Let $E$ be a symmetric $d$-regular EEA over $\R$ with $a_{ij}=c>0$ for
all edges $\{i,j\}$ and $a_{ij}=0$ otherwise.  Then the eigenvalues of
$A$ satisfy $\lambda_1=cd$ and
\[
  \frac{c\,h(E,\mathcal{B})^2}{2d}\;\le\;\lambda_1-\lambda_2\;\le\;
  2c\,h(E,\mathcal{B}).
\]
\end{corollary}

\begin{proof}
Apply~\eqref{eq:cheeger} to $A=cG$, where $G$ is the adjacency matrix of
$\varGamma(E,\mathcal{B})$.
\end{proof}

The two bounds do not match: for the complete graph $K_{10}$, $h=5$ and
$\lambda_1-\lambda_2=10=2h$, while $h^2/(2d)=25/18$.

\begin{corollary}[Alon--Boppana]\label{thm:alon-boppana}
Fix $d\ge2$ and let $(E_m,\mathcal B_m)$ be $d$-regular adjacency
evolution algebras with $\dim E_m\to\infty$.  If $\lambda_2^{(m)}$ is
the second largest eigenvalue of the structural matrix of $E_m$, then
$\liminf_{m\to\infty}\lambda_2^{(m)}\ge2\sqrt{d-1}$.
\end{corollary}

\begin{proof}
This is the Alon--Boppana theorem in the form of Nilli~\cite{Nilli91}
for $d$-regular graphs with a fixed degree $d$ and a number of vertices
tending to infinity; it does not apply to families of growing degree.
\end{proof}

A $d$-regular adjacency evolution algebra with connected graph is called
\emph{Ramanujan} if every eigenvalue $\lambda$ of $A$ with $|\lambda|\ne d$
satisfies $|\lambda|\le2\sqrt{d-1}$, that is, if
$\varGamma(E,\mathcal B)$ is a Ramanujan graph.  By
Proposition~\ref{thm:spectral-C} and~\eqref{eq:cheeger}:

\begin{corollary}\label{thm:Ramanujan-expansion}
Every Ramanujan evolution algebra satisfies
$h(E,\mathcal B)\ge \bigl(d-2\sqrt{d-1}\bigr)/2$.
\end{corollary}

\section{Constructions, tensor products and examples}\label{sec:constructions}

\subsection{Cayley evolution algebras}\label{sec:cayley}

\begin{definition}\label{def:cayley-EA}
Let $G$ be a finite group and $S\subseteq G\setminus\{1\}$ a symmetric
generating set.  The \emph{Cayley evolution algebra}
$\mathcal{C}(G,S,\K)$ has natural basis $\{e_g:g\in G\}$ and structure
constants $a_{g,gs}=1$ for $s\in S$, all other structure constants
being zero.
\end{definition}

Two vertices $g\ne g'$ are adjacent in $\varGamma(\mathcal C(G,S,\K))$
if and only if $g'=gs$ or $g=g's$ for some $s\in S$; since $S=S^{-1}$,
this is adjacency in the Cayley graph $\mathrm{Cay}(G,S)$.  Hence
$\varGamma(\mathcal C(G,S,\K))=\mathrm{Cay}(G,S)$, the algebra
$\mathcal C(G,S,\R)$ is the adjacency evolution algebra of this graph,
and a family of Cayley graphs that is an expander family of fixed degree
$|S|$ yields an EEA family.

\begin{example}\label{ex:SL2-EEA}
Let
\[
  a=\begin{pmatrix}1&3\\0&1\end{pmatrix},\qquad
  b=\begin{pmatrix}1&0\\3&1\end{pmatrix},\qquad
  S=\{a,a^{-1},b,b^{-1}\}\subset\mathrm{SL}_2(\Z),
\]
and let $S_p$ be the reduction of $S$ modulo a prime $p\ge5$.  The four
elements of $S_p$ are distinct, and $S_p$ generates
$\mathrm{SL}_2(\F_p)$, since $3$ is invertible modulo $p$ and suitable
powers of the reductions of $a$ and $b$ are the elementary matrices
$\bigl(\begin{smallmatrix}1&1\\0&1\end{smallmatrix}\bigr)$ and
$\bigl(\begin{smallmatrix}1&0\\1&1\end{smallmatrix}\bigr)$.  The group
$\langle S\rangle$ is a nonabelian free group, hence not virtually
solvable, and therefore Zariski dense in $\mathrm{SL}_2$; by the
Bourgain--Gamburd theorem~\cite{BG08} (see also~\cite{Kowalski19}) the
$4$-regular Cayley graphs $\mathrm{Cay}(\mathrm{SL}_2(\F_p),S_p)$ form an
expander family as $p\to\infty$ through the primes.  Hence the algebras
$\mathcal{C}(\mathrm{SL}_2(\F_p),S_p,\K)$ form an EEA family with
maximum valency $4$.
\end{example}

\begin{example}\label{ex:LPS}
Let $p$ and $q$ be distinct primes with $p\equiv q\equiv 1\pmod{4}$,
$q>2\sqrt p$, and $p$ a quadratic residue modulo $q$.  The
Lubotzky--Phillips--Sarnak construction~\cite{LPS88} associates to the
$p+1$ solutions of $p=a_0^2+a_1^2+a_2^2+a_3^2$ with $a_0>0$ odd and
$a_1,a_2,a_3$ even a symmetric set $S_{p,q}\subset\mathrm{PSL}_2(\F_q)$
of $p+1$ elements, and $\mathrm{Cay}(\mathrm{PSL}_2(\F_q),S_{p,q})$ is a
non-bipartite $(p+1)$-regular Ramanujan graph on $q(q^2-1)/2$ vertices.
The adjacency evolution algebra $\mathcal{C}(\mathrm{PSL}_2(\F_q),S_{p,q},\R)$
is therefore a Ramanujan evolution algebra, and for fixed $p$ these
algebras form an EEA family as $q\to\infty$.
\end{example}

\subsection{Tensor products}\label{sec:tensor}

Let $E_1$ and $E_2$ be evolution algebras over $\K$ with natural bases
$\{e_i\}$ and $\{f_j\}$ and structural matrices $A_1=(a^{(1)}_{ik})$ and
$A_2=(a^{(2)}_{jl})$.  The tensor product $E_1\otimes E_2$, with
multiplication $(x\otimes y)(x'\otimes y')=xx'\otimes yy'$, satisfies
$(e_i\otimes f_j)(e_k\otimes f_l)=0$ for $(i,j)\ne(k,l)$ and
$(e_i\otimes f_j)^2=\sum_{k,l}a^{(1)}_{ik}\,a^{(2)}_{jl}\,e_k\otimes f_l$.
Hence $E_1\otimes E_2$ is an evolution algebra with natural basis
$\{e_i\otimes f_j\}$ and structural matrix the Kronecker product
$A_1\otimes A_2$.  This is the starting point of~\cite{CMMT23}, where the
directed graph of $E_1\otimes E_2$, with a loop at $i$ whenever
$a_{ii}\ne 0$, is identified with the categorical product of the directed
graphs of the factors \cite[Proposition~5.3]{CMMT23}.  The following
proposition reads this in the loopless undirected graph.

\begin{proposition}\label{prop:tensor-graph}
Let $(i,j)\ne(k,l)$.  Then $(i,j)\to(k,l)$ is an arrow of
$D(E_1\otimes E_2)$ if and only if $a^{(1)}_{ik}\ne 0$ and
$a^{(2)}_{jl}\ne 0$.  Assume that the supports of $A_1$ and $A_2$ are
symmetric, and write $\varGamma_t=\varGamma(E_t)$.  Then the support of
$A_1\otimes A_2$ is symmetric, and:
\begin{enumerate}
  \item if all diagonal entries of $A_1$ and $A_2$ vanish,
        $\varGamma(E_1\otimes E_2)$ is the categorical product
        $\varGamma_1\times\varGamma_2$: $(i,j)$ and $(k,l)$ are adjacent
        if and only if $i\sim k$ in $\varGamma_1$ and $j\sim l$ in
        $\varGamma_2$;
  \item if all diagonal entries of $A_1$ and $A_2$ are nonzero,
        $\varGamma(E_1\otimes E_2)$ is the strong product
        $\varGamma_1\boxtimes\varGamma_2$: distinct $(i,j)$ and $(k,l)$
        are adjacent if and only if ($i=k$ or $i\sim k$) and ($j=l$ or
        $j\sim l$).
\end{enumerate}
\end{proposition}

\begin{proof}
Over a field, $a^{(1)}_{ik}a^{(2)}_{jl}\ne 0$ if and only if both factors
are nonzero, which is the first assertion; symmetry of the supports
follows.  In case (i), $a^{(1)}_{ik}\ne 0$ forces $i\ne k$ and $i\sim k$,
and likewise in the second factor.  In case (ii), $a^{(1)}_{ik}\ne 0$
holds if and only if $i=k$ or $i\sim k$, and likewise in the second
factor.
\end{proof}

In particular, the graph of a tensor product is \emph{not} the Cartesian
product $\varGamma_1\,\square\,\varGamma_2$, for which isoperimetric
inequalities are known \cite{ChungTetali98}.  With zero diagonals both
coordinates must move simultaneously, and expansion can be lost.

\begin{theorem}\label{thm:tensor-zero}
Let $E_1$, $E_2$ have structural matrices with symmetric support and
zero diagonal, and assume that $\varGamma_1$ and $\varGamma_2$ are
connected with at least two vertices.
\begin{enumerate}
  \item $\varGamma(E_1\otimes E_2)$ is connected if and only if at least
        one of $\varGamma_1$, $\varGamma_2$ is not bipartite.  In
        particular, if both are bipartite, then $h(E_1\otimes E_2)=0$
        although $h(E_1),h(E_2)>0$.
  \item Assume moreover that $A_t$ is the adjacency matrix of the
        $d_t$-regular graph $\varGamma_t$, with eigenvalues
        $d_t=\mu^{(t)}_1\ge\mu^{(t)}_2\ge\cdots$, and put
        $\lambda^*=\max\{\mu^{(1)}_r\mu^{(2)}_s:(r,s)\ne(1,1)\}$ and
        $\theta_t=\max_{r\ge 2}|\mu^{(t)}_r|/d_t$.  Then
        \[
          h(E_1\otimes E_2)\ge\frac{d_1d_2-\lambda^*}{2}
          \ge\frac{d_1d_2}{2}\bigl(1-\max\{\theta_1,\theta_2\}\bigr).
        \]
        The middle term is positive if and only if at least one of
        $\varGamma_1,\varGamma_2$ is not bipartite, and the right-hand
        term is positive if and only if neither is bipartite.  In
        particular $h(E^{\otimes m})\ge\tfrac12 d^m(1-\theta)$ for all
        $m\ge 1$ when $E_1=E_2=E$ with $\theta=\theta_1$, and this bound
        is positive if and only if $\varGamma(E)$ is not bipartite.
\end{enumerate}
\end{theorem}

\begin{proof}
By Proposition~\ref{prop:tensor-graph}(i),
$\varGamma(E_1\otimes E_2)=\varGamma_1\times\varGamma_2$.

(i) This is Weichsel's theorem~\cite{Weichsel62}, the undirected case of
McAndrew's theorem~\cite{McAndrew63} used in \cite[Theorem~5.4 and
Corollary~5.5]{CMMT23}; we recall the argument.  If
$c_t\colon V(\varGamma_t)\to\Z/2$ are proper $2$-colourings, every edge
of $\varGamma_1\times\varGamma_2$ changes both $c_1$ and $c_2$, so
$c_1(i)+c_2(j)$ is constant on connected components; both values occur,
so the product is disconnected.  Conversely, let $\varGamma_1$ be
non-bipartite.  Since $\varGamma_1$ is connected and contains an odd
cycle, there is $N$ such that $\varGamma_1$ contains a walk of every
length $m\ge N$ between any two vertices.  In $\varGamma_2$ there is a
walk of some length $L$ from $j$ to $l$, hence walks of every length
$L+2t$, $t\ge 0$.  Choosing $m\ge\max\{N,L\}$ with $m\equiv L$ mod $2$
and pairing walks of length $m$ in the two factors gives a walk from
$(i,j)$ to $(k,l)$ in $\varGamma_1\times\varGamma_2$.

(ii) The matrix $A_1\otimes A_2$ is the adjacency matrix of the
$d_1d_2$-regular graph $\varGamma_1\times\varGamma_2$, and its
eigenvalues are the products $\mu^{(1)}_r\mu^{(2)}_s$, the value $d_1d_2$
coming from $(r,s)=(1,1)$.  Hence its second largest eigenvalue, counted
with multiplicity, is $\lambda^*$, and the left-hand inequality
in~\eqref{eq:cheeger}, whose proof uses only regularity, gives the first
bound.  If $(r,s)\ne(1,1)$, then
$|\mu^{(1)}_r\mu^{(2)}_s|\le\max\{\theta_1,\theta_2\}\,d_1d_2$, which gives
the second.  Since $\varGamma_t$ is connected, $\mu^{(t)}_r<d_t$ for
$r\ge2$ (Proposition~\ref{thm:spectral-C}), and $\mu^{(t)}_r=-d_t$ for
some $r$ if and only if $\varGamma_t$ is bipartite: the signed indicator
vector of a bipartition is an eigenvector for $-d_t$, and conversely, if
$G_tx=-d_tx$ with $x\ne0$ and $|x_i|$ is maximal, then $x_j=-x_i$ for
every neighbour $j$ of $i$, and propagating along paths shows that the
sign of $x$ is a proper $2$-colouring.  Hence
$\lambda^*=d_1d_2$ exactly when both factors are bipartite, and
$\theta_t=1$ exactly when $\varGamma_t$ is bipartite.  For tensor powers
the same argument applies to $A^{\otimes m}$, whose eigenvalues are the
products $\mu_{r_1}\cdots\mu_{r_m}$.
\end{proof}

For $e_1^2=e_2$, $e_2^2=e_1$ ($\varGamma=K_2$, $h=1$), the graph of
$E\otimes E$ consists of two disjoint edges, so $h(E\otimes E)=0$.  Thus
no bound of the form $h(E_1\otimes E_2)\ge\varphi(h(E_1),h(E_2))$ with
$\varphi>0$ holds for zero diagonals.  Nonzero diagonal entries, in
contrast, make expansion stable.

\begin{theorem}\label{thm:tensor-unit}
Let $E_1$, $E_2$ have structural matrices with symmetric support and
nonzero diagonal, and assume that $\varGamma_1$ and $\varGamma_2$ are
connected.
\begin{enumerate}
  \item $\varGamma(E_1\otimes E_2)=\varGamma_1\boxtimes\varGamma_2$ is
        connected, and $E_1\otimes E_2$ again has a structural matrix
        with symmetric support and nonzero diagonal.  Consequently every
        tensor product of finitely many such algebras has connected
        graph.
  \item Assume moreover that $A_t=I+G_t$, where $G_t$ is the adjacency
        matrix of the $d_t$-regular graph $\varGamma_t$ on $n_t$
        vertices, with eigenvalues $d_t=\mu^{(t)}_1\ge\mu^{(t)}_2\ge\cdots$,
        and put $\theta'_t=\max_{r\ge 2}|1+\mu^{(t)}_r|/(d_t+1)$.  Then
        $\theta'_t\le\max\{(1+\mu^{(t)}_2)/(d_t+1),\,(d_t-1)/(d_t+1)\}<1$
        and
        \[
          h(E_1\otimes E_2)\ge\tfrac12\bigl(1-\max\{\theta'_1,\theta'_2\}\bigr)
          (d_1+1)(d_2+1).
        \]
        In particular $h(E^{\otimes m})\ge\tfrac12(1-\theta')(d+1)^m$ for
        all $m\ge 1$ when $E_1=E_2=E$ with $\theta'=\theta'_1$; here
        $E^{\otimes m}$ has dimension $n^m$ and its graph is
        $((d+1)^m-1)$-regular.
\end{enumerate}
\end{theorem}

\begin{proof}
(i) The first assertion is Proposition~\ref{prop:tensor-graph}(ii).
For connectivity, join $i$ to $k$ by a path $u_0,\dots,u_p$ in
$\varGamma_1$ and $j$ to $l$ by a path $w_0,\dots,w_q$ in
$\varGamma_2$, and extend the shorter one by repeating its last vertex; consecutive pairs $(u_t,w_t)$,
$(u_{t+1},w_{t+1})$ are equal or adjacent in the strong product, so the
strong product is connected (this also follows from
\cite[Corollary~5.5]{CMMT23}).  The diagonal entries
$a^{(1)}_{ii}a^{(2)}_{jj}$ of $A_1\otimes A_2$ are nonzero, and induction
on the number of factors gives the last assertion.

(ii) For $r\ge 2$ we have $-d_t\le\mu^{(t)}_r\le\mu^{(t)}_2<d_t$ by
Proposition~\ref{thm:spectral-C}, so
$1-d_t\le1+\mu^{(t)}_r\le1+\mu^{(t)}_2<1+d_t$, which gives the bound on
$\theta'_t$.  The matrix $(I+G_1)\otimes(I+G_2)$ has entries in $\{0,1\}$
and unit diagonal, so by Proposition~\ref{prop:tensor-graph}(ii) it
equals $I+G$, where $G$ is the adjacency matrix of the strong product, a
graph on $n_1n_2$ vertices that is $D$-regular with
$D=(d_1+1)(d_2+1)-1$.  The eigenvalues of $G$ are
$(1+\mu^{(1)}_r)(1+\mu^{(2)}_s)-1$; the pair $(1,1)$ gives $D$, and since
$|1+\mu^{(t)}_r|\le d_t+1$ for all $r$, every other pair gives at most
$\max\{\theta'_1,\theta'_2\}(d_1+1)(d_2+1)-1$.  Therefore
$D-\lambda_2(G)\ge(1-\max\{\theta'_1,\theta'_2\})(d_1+1)(d_2+1)$, and the
left-hand inequality in~\eqref{eq:cheeger} applied to the $D$-regular
graph gives the bound.  For $E^{\otimes m}$ the same argument applies to
$A^{\otimes m}=I+G^{(m)}$, whose eigenvalues are
$\prod_{t=1}^m(1+\mu_{r_t})$; every product with
$(r_1,\dots,r_m)\ne(1,\dots,1)$ has absolute value at most
$\theta'(d+1)^m$.
\end{proof}

\begin{example}\label{ex:tensor-K2}
Let $e_1^2=e_2^2=e_1+e_2$, so that $A=I+G$ with $G$ the adjacency
matrix of $K_2$ ($d=1$, $\mu_2=-1$, $\theta'=0$).  Then $A^{\otimes m}$
is the all-ones matrix of size $2^m$, $\varGamma(E^{\otimes m})=K_{2^m}$
and $h(E^{\otimes m})=2^{m-1}$, so the bound of
Theorem~\ref{thm:tensor-unit}(ii) is attained.  For the $4$-cycle with
unit diagonal ($\mu=2,0,0,-2$), $\theta'=1/3$ and the bound for $m=2$
equals $3=h(C_4\boxtimes C_4)$; the minimum is attained by two
consecutive rows of the $4\times4$ torus.
\end{example}

The algebras $E^{\otimes m}$ in Theorem~\ref{thm:tensor-unit} form an EEA
family with $h\to\infty$; their degrees are unbounded, so their graphs
are not an expander family in the bounded-degree sense.

\subsection{Examples}

\begin{example}[Cycle algebras]\label{ex:cycle}
Let $n\ge 3$ and let $E_n$ have natural basis $\{e_0,\dots,e_{n-1}\}$
and $e_i^2=e_{i-1}+e_{i+1}$ (indices mod $n$).  The underlying graph is
the cycle $C_n$ and $h(E_n,\mathcal{B})=2/\lfloor n/2\rfloor\to 0$
\cite[Example~3.1.3(2)]{Kowalski19}.  Each $E_n$ is an EEA, but
$(E_n)_{n\ge 3}$ is not an EEA family.
\end{example}

\begin{example}[Complete algebras]\label{ex:complete}
Let $n\ge 2$ and $e_i^2=\sum_{j\ne i}e_j$.  The underlying graph is
$K_n$.  A set $W$ with $|W|=w\le n/2$ has $|E(W)|=w(n-w)$, so
$h(E_{K_n},\mathcal{B})=\lceil n/2\rceil$.  These algebras have the
largest possible expansion, but their degree $n-1$ is unbounded.
\end{example}

\begin{example}[Petersen algebra]\label{ex:Petersen}
Let $E_P$ be the adjacency evolution algebra of the Petersen graph, a
$3$-regular graph on $10$ vertices with girth $5$ and spectrum
$\{3^1,1^5,(-2)^4\}$.  Since $1$ and $2$ do not exceed
$2\sqrt2\approx 2.83$, $E_P$ is a Ramanujan evolution algebra.

We have $h(E_P,\mathcal{B})=1$.  Indeed, for $W$ with $|W|\le5$ let
$e(W)$ be the number of edges with both endpoints in $W$; then
$|E(W)|=3|W|-2e(W)$.  A graph on at most $5$ vertices with at least as
many edges as vertices contains a cycle, which in the Petersen graph has
length at least $5$; so either $e(W)\le|W|-1$, or $|W|=5$ and the induced
subgraph is a $5$-cycle with no chord (a chord would create a shorter
cycle), whence $e(W)=5$.  In both cases $e(W)\le|W|$, so
$|E(W)|\ge|W|$, with equality for the outer pentagon.

Corollary~\ref{thm:Ramanujan-expansion} gives only
$(3-2\sqrt{2})/2\approx 0.086$, while the spectral gap is $2=2h$, so the
upper bound in Corollary~\ref{thm:cheeger-evop} is attained.  The
Petersen graph is not bipartite, and Theorem~\ref{thm:tensor-zero}(ii)
with $\lambda^*=(-2)(-2)=4$ gives $h(E_P\otimes E_P)\ge 5/2$.
\end{example}

\section{Open problems}\label{sec:open}

\begin{problem}
Let $E_1$, $E_2$ have structural matrices with symmetric support and
nonzero diagonal, but arbitrary weights and degrees.  Is there an
absolute constant $c>0$ with
$h(E_1\otimes E_2)\ge c\min\{h(E_1),h(E_2)\}$?  By
Proposition~\ref{prop:tensor-graph}(ii) the graph of $E_1\otimes E_2$
contains the Cartesian product $\varGamma_1\,\square\,\varGamma_2$ as a
spanning subgraph, so the isoperimetric inequalities
of~\cite{ChungTetali98} are relevant.
\end{problem}

\begin{problem}
Proposition~\ref{thm:ideals} and Theorem~\ref{thm:support-growth} use
symmetry of the support of $A$.  Find a directed expansion condition on
$D(E,\mathcal{B})$ that yields quantitative analogues of these results
for evolution algebras with non-symmetric support.
\end{problem}

\begin{problem}\label{prob:singular-basis}
For singular $E$ natural bases need not be related by monomial
matrices.  Determine whether $h(E,\mathcal{B})$ can depend on
$\mathcal{B}$, and if so, whether $\sup_{\mathcal{B}}h(E,\mathcal{B})$ is
a useful invariant.
\end{problem}

\begin{problem}
Example~\ref{ex:support-fail}(iii) shows that the conclusion of
Theorem~\ref{thm:support-growth} can fail in positive characteristic even
for structure constants in $\{0,1\}$ with nonzero diagonal.  For
$\K=\F_q$, describe the supports of the plenary powers of an EEA in
terms of $\varGamma(E,\mathcal{B})$ and the structure constants.
\end{problem}

\begin{problem}\label{prob:fusion}
Let $E$ be a symmetric Markov evolution algebra with connected graph and
$u=e_1+\dots+e_n$.  By Proposition~\ref{prop:u}, $E$ decomposes into
eigenspaces $E_\lambda$ of $L_u$, with $E_1=\R u$.  Determine the fusion
law of $u$, that is, the smallest sets $\lambda\star\mu$ of eigenvalues
with $E_\lambda E_\mu\subseteq\bigoplus_{\nu\in\lambda\star\mu}E_\nu$, in
terms of $\varGamma(E,\mathcal B)$.  Which evolution algebras are
generated by primitive idempotents obeying a fixed finite fusion law, in
the sense of \cite{HRS15a,GS26}?
\end{problem}

\begin{problem}
Lubotzky~\cite{Lub14} surveys higher-dimensional expander complexes.
Define higher-order EEAs based on simplicial complexes, and determine
which results of this paper extend to that setting.
\end{problem}

\end{document}